\documentclass[11 pt,a4paper,twoside,reqno]{amsart}
\usepackage{amsfonts,amssymb,amscd,amsmath,enumerate,verbatim,calc}

\usepackage{graphicx}
\usepackage{amsthm,latexsym}
\usepackage{a4}

\newtheorem{theorem}{Theorem}[section]

\newtheorem{corollary}[theorem] {Corollary}
\newtheorem{definition}[theorem]{Definition}

\newtheorem{proposition}[theorem]{Proposition}
\newtheorem{remark}[theorem]{Remark}

\theoremstyle{questions}
\newtheorem*{questions}{Questions}

\def\det{{\rm det}}
\def\Tr{{\rm Tr}}

\def\sin{{\rm sin}}
\def\cos{{\rm cos}}
\def\tan{{\rm tan}}
\def\cosec{{\rm cosec}}
\def\sec{{\rm sec}}
\def\cot{{\rm cot}}
\def\Or{{\rm Or}}
\def\Sym{{\rm Sym}}
\def\PSym{{\rm PSym}}
\def\S{{\rm S}}

\def\Stab{{\rm Stab}}

\newcommand{\ncom}{\newcommand}
\ncom{\mylabel}[1]{{\rm (#1)}\label{#1}}
\ncom{\Hom}{{\textit{Hom}}}
\ncom{\eop}{{\hfill $\Box$}}
\begin{document}
\baselineskip=16pt

\setcounter{tocdepth}{1}

\title[Symmetric matrices of trace zero]{Dynamics of two by two symmetric matrices of trace zero}

\author{Arijit Mukherjee}
\address{Department of Mathematics, Indian Institute of Science Education and Research Tirupati, Karakambadi Rd,  opp. Sree Rama Engineering College, Rami Reddy Nagar, Mangalam, Tirupati, Andhra Pradesh - 517 507, India.}
\email{mukherjee90.arijit@gmail.com}

\begin{abstract}  
In this paper, we describe the entire structure of the vector space $\Sym_2^0$ of all symmetric matrices of size $2$ having trace zero.  This is motivated by the geometrical interpretation of any arbitrary element of $\Sym_2^0$.  We further study the orbits and stable sets of these elements.  As an application of the obtained structure of $\Sym_2^0$, we obtain the symmetric matrices of size $2$, trace of whose product with any trace zero symmetric matrix is zero.  Finally some well known trigonometric formulas are interpreted geometrically incorporating the anatomy of $\Sym_2^0$. 
\end{abstract}
 
\maketitle




\section{Introduction}

The study of symmetric matrices and trace zero matrices attracted considerable attention.  Many mathematicians have researched on symmetric matrices to study SNIEP, symmetric non-negative inverse eigenvalue problem (cf. \cite{Fl},\cite{Lw-Lo},\cite{Lw-Mc} \& \cite{Mc-Ne}).  Also people have independently worked on trace zero matrices and found necessary and sufficient conditions for a matrix to have zero trace (cf. \cite{Al-Mh} \& \cite{Sh}).  A study of SNIEP for trace zero symmetric matrices can be found in \cite{Sp}.

This paper is devoted to the study of trace zero symmetric matrices of size $2$ and its applications.  But we do it in a different context and therefore follow a different approach altogether.  

We begin by fixing some notations which we are going to use repeatedly.  Let $\mathbb{R}$ be the field of real numbers.  In this paper, by a vector space we always mean a vector space over $\mathbb{R}$ and by a matrix we mean a matrix with real entries.  Let $\Sym_n$ \& $\mathcal{O}_n$ be the set of all symmetric matrices and orthogonal matrices of size $n$ respectively and $\Sym^0_n$ be the subset of $\Sym_n$ consisting of all matrices having trace zero. By $\PSym^0_n$ we denote the set of all $n\times n$ symmetric matrices, trace of whose product with any element of $\Sym^0_n$ is zero.  We reserve the notations $\mathbb{I}_n$ and $\mathbb{O}_n$ for identity matrix and zero matrix of size $n$ respectively.  We denote the trace of a given matrix $A$ by $\Tr(A)$ and determinant of $A$ by $\det(A)$.  Given any two matrices $A$ and $B$ of size $n$, we denote the matrix multiplication $A\cdot B$ simply by their juxtaposition $AB$.

In this paper, our main aim is to describe the precise structure of $\Sym^0_2$.  In Theorem \ref{structure of Sym^0_2}, we show that the elements of $\Sym^0_2$ are precisely of the form $\big(\begin{smallmatrix}
\lambda\; \cos(\theta) & \lambda\; \sin(\theta)\\
\lambda\; \sin(\theta) & -\lambda\; \cos(\theta)
\end{smallmatrix}\big)$ for some $\lambda \in \mathbb{R}, \theta \in [0,2\pi)$.  Using the structure of $\Sym^0_2$, we further show that the set of all trace zero symmetric matrices of size $2$ having eigenvalues $1$ and $-1$ is same as the set of all size $2$ orthogonal matrices having determinant $-1$, (cf. Corollary \ref{equality of some orthogonal matrices and some symmetric matrices of trace zero}).

We come up with the geometrical interpretation of the elements of $\Sym_2^0$ as well.  In fact, this is the motivation behind finding the anatomy of $\Sym_2^0$.  Moreover, we extensively discuss about the orbits,  raise the questions about the finiteness of those orbits and answer those using the obtained structure of $\Sym_2^0$.  We obtain a necessary and sufficient condition for finiteness of the orbit $\Or_{(a,b)}(T_{\theta}^{\lambda})$ of an arbitrary element $T_{\theta}^{\lambda}$ of $\Sym_2^0$ starting at a point $(a,b)$ (cf. Theorem \ref{orbit finiteness_NASC conditions}).

We also study the dynamics of $T^{\lambda}_{\theta}$ and show how the stable set $\Stab_{T_{\theta}^{\lambda}}((a,b))$ of any point $(a,b)$ with respect to $T^{\lambda}_{\theta}$ varies as $\lambda$ varies.  In the process, we obtain that the stable set $\Stab_{T_{\theta}^{\lambda}}((a,b))$ either contains only $(a,b)$ or is the whole of $\mathbb{R}^2$ depending on whether $\lambda$ lies in the open interval $(1,1)$ or not, (cf. Theorem \ref{stable set}).   

In section \ref{section 3}, we look upon a couple of applications of the structure of $\Sym_2^0$.  Firstly, we derive the structure of $\PSym_2^0$ in subsection \ref{subsection 3.1}.  To be precise, we show that $\PSym_2^0$ consists of scalar matrices and scalar matrices only, (cf. Theorem \ref{main theorem_special case}).  We then prove that the obtained anatomy of $\PSym_2^0$ can be generalised for any $n\geq 2$ using Frobenius inner product on $\Sym_n$ (cf. Theorem \ref{main theorem_with more prerequisite}).  

As another application, we talk about two rigid motions, namely rotation and reflection, of any point of the Euclidean plane and show that rotating the point of reflection of a given point with respect to a line is same as reflecting it with respect to some other line, (cf. Theorem \ref{reinterpreting some trigonometric formulas}).  Though this can be proved using simple techniques of Euclidean geometry, but we do it incorporating the structure of $\Sym_2^0$ and as a result the elucidation seems to be an elegant one.  The given proof can also be thought of as a geometric interpretation of couple of well known and frequently used trigonometric formulas (cf. Remark \ref{trigonometry to geometry}).

In Section \ref{section 4}, we conclude by indicating that the structure of some subsets of $\Sym^0_n$ can be obtained for $n>2$ as well by adapting the method used in Theorem \ref{structure of Sym^0_2}, provided some conditions being suitably put on the set of eigenvalues of its elements.

\section{On the structure of $\S\lowercase{ym}^0_2$ and orbits of its elements}\label{section 2} 
In this section, we provide the structure of $\Sym^0_2$, analyse its elements from a geometric viewpoint and then discuss upon the orbits and the stable sets of those elements.  

For any given real number $\theta$, denote the line in $\mathbb{R}^2$ passing through origin and making an angle $\theta$ with the positive direction of $x$-axis in the anticlockwise direction by $L_{\theta}$.  Given any $\lambda\in \mathbb{R}$, define a map (geometrically), denoted by $T_{\theta}^{\lambda}$, as follows : Given any point $(x_1,x_2)$, the map $T_{\theta}^{\lambda}$ first reflects the point $(x_1,x_2)$ with respect to the line $L_{\theta}$ and scales that by $\lambda$ followed by that.  Denoting the reflection of $(x_1,x_2)$ with respect to the line $L_{\theta}$ by $R_{\theta}((x_1,x_2))$, the map $T_{\theta}^{\lambda}$ can be given as follows:
\begin{equation}\label{geometric definition of elements of Sym_0^2}
\begin{split}
T_{\theta}^{\lambda}:\mathbb{R}^2&\rightarrow \mathbb{R}^2\\
(x_1,x_2)&\mapsto \lambda R_{\theta}((x_1,x_2)).
\end{split}
\end{equation}   

Define the orbit $\Or_{(a,b)}(T_{\theta}^{\lambda})$ of the map $T_{\theta}^{\lambda}$ starting at a point $(a,b)$ as follows:
\begin{equation*}
\Or_{(a,b)}(T_{\theta}^{\lambda}):=\{(T_{\theta}^{\lambda})^n((a,b))\mid n\in \mathbb{N}\}.
\end{equation*} 

We now ask the following questions:
\begin{questions}
\begin{enumerate}
\item Given any $(a,b)\in \mathbb{R}^2$, for what values of $\theta$ and $\lambda$,the orbit $\Or_{(a,b)}(T_{\theta}^{\lambda})$ is finite?
\item For what values of $\theta$ and $\lambda$, the orbit $\Or_{(a,b)}(T_{\theta}^{\lambda})$ is a singleton set, where $(a,b)\in \mathbb{R}^2$?
\item Given any $(a,b)\in \mathbb{R}^2$, for what values of $\theta$ and $\lambda$, the orbit $\Or_{(a,b)}(T_{\theta}^{\lambda})$ is a set having two points?
\item For what values of $\theta$ and $\lambda$,  $(T_{\theta}^{\lambda})^n=\mathbb{I}_2$ for some positive integer $n$?
\item Given any $(a,b)\in \mathbb{R}^2$, the sequence $\{(T_{\theta}^{\lambda})^n((a,b))\}_{n\geq 0}$ is convergent in usual topology and in discrete topology for what values of $\theta$ and $\lambda$?
\end{enumerate}
\end{questions}

To answer these questions, we first calculate $T_{\theta}^{\lambda}(P_0)$ for any given point $P_0\in \mathbb{R}^2$.  We do the calculation assuming that both the coordinates of the point $P_0=(x_0,y_0)$ are positive and the line joining $P_0$ and origin makes an angle $\alpha$ with the positive direction of $x$-axis.  That is to say, $P_0$ lies in the line $L_{\alpha}$.  Moreover, we assume that $0<\alpha\leq \theta\leq 2\theta-\alpha<\tfrac{\pi}{2}$, that is to say the lines $L_{\alpha}$, $L_{\theta}$ and $L_{2\theta-\alpha}$ are having slopes in non-decreasing order, lies in first and third quadrant and none of those are $x$-axis or $y$-axis.  The calculations are very much similar in the remaining cases as well.

As we mentioned, we start with a point $P_0=(x_0,y_0)$ lying in the line $L_{\alpha}$.  We first want to determine the reflection of $P_0$ with respect to the line $L_{\theta}$.  For that we draw a perpendicular from the point $P_0$ to the line $L_{\theta}$ and denote the point of intersection of this with $L_{\theta}$ by $A$.  Then we extend the line segment $\overline{P_0A}$ and suppose that intersects the line $L_{2\theta-\alpha}$ at the point $P_1=(x_1,y_1)$.  Then clearly the point $P_1$ is the reflection of $P_0$.  Now we determine $x_1$ and $y_1$ in terms of the known quantities $x_0$, $y_0$, $\theta$ and $\alpha$ using the properties of reflection.

\begin{figure}
\centering
\includegraphics[scale=.45]{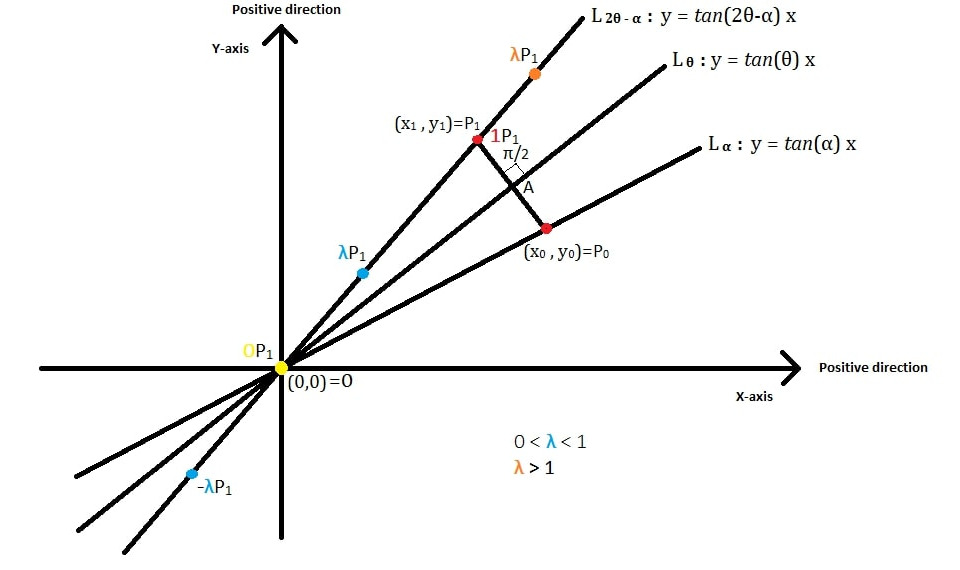}
\caption{Image of a point $P_0$ under the map $\lambda R_{\theta}$ for some values of $\lambda$} \label{sample-figure_1}
\end{figure}

The equations of the line $L_{\theta}$, $L_{\alpha}$ and $L_{2\theta-\alpha}$ are given as follows:
\begin{equation*}
L_{\theta} : y=\tan(\theta)\; x,\; L_{\alpha} : y=\tan(\alpha)\; x,\; L_{2\theta-\alpha} : y=\tan(2\theta-\alpha)\; x.
\end{equation*}
By properties of reflection, we get that the triangle $AP_0O$ and $AP_1O$ are congruent to each other.  That is to say, $\bigtriangleup AP_0O\cong \bigtriangleup AP_1O$.  Therefore, the length $\mathcal{L}(\overline{OP_0})$ of the line segment $\overline{OP_0}$ is equal to that of $\overline{OP_1}$, that is to say, $\mathcal{L}(\overline{OP_0})=\mathcal{L}(\overline{OP_1})$. We then have the following implications.
\begin{equation*}
\begin{split}
\mathcal{L}(\overline{OP_0})&=\mathcal{L}(\overline{OP_1})\Rightarrow \sqrt{x_0^2+y_0^2}=\sqrt{x_1^2+y_1^2}\\&\Rightarrow \sqrt{x_0^2+x_0^2\;\tan^2(\alpha)}=\sqrt{x_1^2+x_1^2\;\tan^2(2\theta-\alpha)}\\&\Rightarrow x_0\;\sec(\alpha)=x_1\;\sec(2\theta-\alpha)\Rightarrow x_1=x_0\;\sec(\alpha)\cos(2\theta-\alpha).
\end{split}
\end{equation*}
So the abscissa $x_1$ of the point of reflection $P_1$ of $P_0$ is given as follows:
\begin{equation}\label{abscissa of reflection}
x_1=x_0\;\sec(\alpha)\cos(2\theta-\alpha).
\end{equation}
As $(x_1,y_1)$ lies in the line $L_{2\theta-\alpha}$, we have $y_1=\tan(2\theta-\alpha)\; x_1$.  We now have the following implications:
\begin{equation*}
\begin{split}
y_1=\tan(2\theta-\alpha)\; x_1&\Rightarrow y_1=x_0\;\sec(\alpha)\cos(2\theta-\alpha)\tan(2\theta-\alpha) \;\;(\text{by}\; \eqref{abscissa of reflection})\\
&\Rightarrow y_1=x_0\;\sec(\alpha)\sin(2\theta-\alpha)\\
&\Rightarrow y_1=y_0\;\cosec(\alpha)\sin(2\theta-\alpha)\;\;( \text{as}\; (x_0,y_0)\;\text{lies in}\; L_{\alpha}).
\end{split}
\end{equation*}
So the ordinate $y_1$ of the point of reflection $P_1$ of $P_0$ is given as follows:
\begin{equation}\label{ordinate of reflection}
y_1=y_0\;\cosec(\alpha)\sin(2\theta-\alpha).
\end{equation}

\begin{remark}\label{points on line of reflection are invariant under reflection}
Taking $\theta=\alpha$ in equations \eqref{abscissa of reflection} and \eqref{ordinate of reflection}, we have $P_0=(x_0,y_0)=(x_1,y_1)=P_1$.  On the other hand, if $\theta\neq\alpha$ and $P_0\neq (0,0)$, then $P_0\neq P_1$ as $P_0\in L_{\alpha}$ and $P_1\in L_{2\theta-\alpha}$.  This can be justified also using \eqref{abscissa of reflection} and \eqref{ordinate of reflection}.  This depicts the fact that reflection of any point located on the axis of the reflection is the point itself.  Moreover, this is true exclusively for points located on the line of reflection.
\end{remark}

So we have obtained the coordinates of the point of reflection of $(x_0,y_0)$ about the line $L_{\theta}$.  For the moment, we take a break from the discussion about the map $T^{\lambda}_{\theta}$.  We continue with the same  after a while and answer the questions raised.

We now provide the description of $\Sym_2^0$.  Towards that, we have the following proposition which says about the structure of the set $\mathcal{O}_2$ of orthogonal matrices of size $2$.  Though this is a standard result (for example, see \cite[p. 348]{He} for determinant $1$ orthogonal matrices), we include this over here for the sake of continuity.

\begin{proposition}\label{structure of orthogonal matrices of size 2}
The collection $\mathcal{O}_2$ of all orthogonal matrices of size $2$ is given as follows:
\begin{equation*}
\mathcal{O}_2= \lbrace \big(\begin{smallmatrix}
 \cos(\alpha) &  \sin(\alpha)\\
 -\sin(\alpha) &  \cos(\alpha)
\end{smallmatrix}\big)\mid  \alpha \in [0,2\pi) \rbrace \cup \lbrace \big(\begin{smallmatrix}
 \cos(\beta) &  \sin(\beta)\\
 \sin(\beta) &  -\cos(\beta)
\end{smallmatrix}\big)\mid  \beta \in [0,2\pi) \rbrace.
\end{equation*}
\end{proposition}
\begin{proof}
Given $x=(x_1,x_2)$ and $y=(y_1,y_2)$, by $\langle x, y\rangle$ we denote the standard inner product of $x$ and $y$ in $\mathbb{R}^2$.  That is to say, $\langle x, y\rangle=x_1y_1+x_2y_2$.  Let $M=\big(\begin{smallmatrix}
p & q\\
r & s
\end{smallmatrix}\big)$ be a given orthogonal matrix.  As the columns of $M$ are of unit norm with respect to the standard inner product of $\mathbb{R}^2$, we have 
\begin{equation}\label{orthogonal matrix_1}
p^2+r^2=1, q^2+s^2=1.
\end{equation}
Then \eqref{orthogonal matrix_1} in turn imply that 
\begin{equation}\label{orthogonal matrix_2}
\mid p \mid \leq 1, \;\mid r \mid \leq 1,\;\mid q \mid \leq 1,\;\mid s \mid \leq 1.
\end{equation}
As the columns of $M$ are orthogonal with respect to the standard inner product of $\mathbb{R}^2$, we have 
\begin{equation}\label{orthogonal matrix_3}
pq+rs=0.
\end{equation}
\textit{Case - 1}:\\
If $\det(A)=1$, then we moreover have 
\begin{equation}\label{orthogonal matrix_4a}
ps-rq=1.
\end{equation}
So, conditions \eqref{orthogonal matrix_2}, \eqref{orthogonal matrix_3} and \eqref{orthogonal matrix_4a} imply that, there exists $\alpha\in[0,2\pi)$ such that 
\begin{equation}\label{orthogonal matrix structure_1}
M=\begin{pmatrix}
p & q\\
r & s
\end{pmatrix}=\begin{pmatrix}
\cos(\alpha) & \sin(\alpha)\\
-\sin(\alpha) & \cos(\alpha)
\end{pmatrix}.
\end{equation}
\textit{Case - 2}:\\
If $\det(A)=-1$, then we moreover have 
\begin{equation}\label{orthogonal matrix_4b}
ps-rq=-1.
\end{equation}
So, conditions \eqref{orthogonal matrix_2}, \eqref{orthogonal matrix_3} and \eqref{orthogonal matrix_4b} imply that, there exists $\beta\in[0,2\pi)$ such that 
\begin{equation}\label{orthogonal matrix structure_2}
M=\begin{pmatrix}
p & q\\
r & s
\end{pmatrix}=\begin{pmatrix}
\cos(\beta) & \sin(\beta)\\
\sin(\beta) & -\cos(\beta)
\end{pmatrix}.
\end{equation}
Therefore, we have the result from \eqref{orthogonal matrix structure_1} and \eqref{orthogonal matrix structure_2}.
\end{proof}

Following theorem describes the entire structure of $\Sym^0_2$.
\begin{theorem}\label{structure of Sym^0_2}
The collection $\Sym^0_2$ of all trace zero symmetric matrices of size $2$ is given as follows:
\begin{equation*}
\Sym^0_2=\lbrace \big(\begin{smallmatrix}
\lambda\; \cos(\theta) & \lambda\; \sin(\theta)\\
\lambda\; \sin(\theta) & -\lambda\; \cos(\theta)
\end{smallmatrix}\big)\mid \lambda \in \mathbb{R}, \theta \in [0,2\pi) \rbrace.
\end{equation*}
\end{theorem}
\begin{proof}
By spectral theorem, we have that any real symmetric matrix is orthogonally diagonalisable and vice versa (cf. \cite[p.347]{He}).  That is, for $n=2$, given any symmetric matrix $A=\big(\begin{smallmatrix}
a & b\\
b & c
\end{smallmatrix}\big)$, there exists an orthogonal matrix $A_O$ such that \begin{equation}\label{towards the structure of Sym^0_2}
A_O  \big(\begin{smallmatrix}
a & b\\
b & c
\end{smallmatrix}\big) {A_O}^{-1}= \big(\begin{smallmatrix}
\gamma & 0\\
0 & \delta
\end{smallmatrix}\big).
\end{equation}
Moreover, if $A\in \Sym^0_2$ then sum of eigenvalues of $A$ is zero, that is, $\gamma+\delta=0$. \\
\textit{Case - 1}:\\
If $\det(A_O)=1$, then by Proposition \ref{structure of orthogonal matrices of size 2} there exists $\alpha\in [0,2\pi)$ such that $A_O=\big(\begin{smallmatrix}
\cos(\alpha) & \sin(\alpha)\\
-\sin(\alpha) & \cos(\alpha)
\end{smallmatrix}\big)$.  So, we also have ${A_O}^{-1}=\big(\begin{smallmatrix}
\cos(\alpha) & -\sin(\alpha)\\
\sin(\alpha) & \cos(\alpha)
\end{smallmatrix}\big)$.
Therefore, from \eqref{towards the structure of Sym^0_2} we have the following:
\begin{equation*}
\begin{split}
A=\begin{pmatrix}
a & b\\
b & c
\end{pmatrix}&=\begin{pmatrix}
\cos(\alpha) & -\sin(\alpha)\\
\sin(\alpha) & \cos(\alpha)
\end{pmatrix} \begin{pmatrix}
\gamma & 0\\
0 & -\gamma
\end{pmatrix}  \begin{pmatrix}
\cos(\alpha) & \sin(\alpha)\\
-\sin(\alpha) & \cos(\alpha)
\end{pmatrix}\\
&=\begin{pmatrix}
\gamma\;\cos^2(\alpha)-\gamma\;\sin^2(\alpha) & 2\gamma\; \cos(\alpha) \sin(\alpha)\\
2\gamma\; \cos(\alpha) \sin(\alpha) & \gamma\;\sin^2(\alpha)-\gamma\;\cos^2(\alpha)
\end{pmatrix}\\
&=\begin{pmatrix}
\gamma\;\cos (2\alpha) & \gamma\;\sin (2\alpha)\\
\gamma\;\sin (2\alpha) & -\gamma\;\cos (2\alpha)
\end{pmatrix}.
\end{split}
\end{equation*}
\textit{Case - 2}:\\
If $\det(A_O)=-1$, then by Proposition \ref{structure of orthogonal matrices of size 2} there exists $\beta\in [0,2\pi)$ such that $A_O=\big(\begin{smallmatrix}
\cos(\beta) & \sin(\beta)\\
\sin(\beta) & -\cos(\beta)
\end{smallmatrix}\big)$.  So, we also have ${A_O}^{-1}=A_O=\big(\begin{smallmatrix}
\cos(\beta) & \sin(\beta)\\
\sin(\beta) & -\cos(\beta)
\end{smallmatrix}\big)$.
Therefore, from \eqref{towards the structure of Sym^0_2} we have the following:
\begin{equation*}
\begin{split}
A=\begin{pmatrix}
a & b\\
b & c
\end{pmatrix}&=\begin{pmatrix}
\cos(\beta) & \sin(\beta)\\
\sin(\beta) & -\cos(\beta)
\end{pmatrix} \begin{pmatrix}
\gamma & 0\\
0 & -\gamma
\end{pmatrix}  \begin{pmatrix}
\cos(\beta) & \sin(\beta)\\
\sin(\beta) & -\cos(\beta)
\end{pmatrix}\\
&=\begin{pmatrix}
\gamma\;\cos^2(\beta)-\gamma\;\sin^2(\beta) & 2\gamma\; \cos(\beta) \sin(\beta)\\
2\gamma\; \cos(\beta) \sin(\beta) & \gamma\;\sin^2(\beta)-\gamma\;\cos^2(\beta)
\end{pmatrix}\\
&=\begin{pmatrix}
\gamma\;\cos (2\beta) & \gamma\;\sin (2\beta)\\
\gamma\;\sin (2\beta) & -\gamma\;\cos (2\beta)
\end{pmatrix}.
\end{split}
\end{equation*}
By taking $\theta=2\alpha$ if $\alpha<\pi$ \& $\theta=2(\alpha-\pi)$ if $\alpha\geq \pi$ in Case - 1 and $\theta=2\beta$ if $\beta<\pi$ \& $\theta=2(\beta-\pi)$ if $\beta\geq \pi$ in Case - 2, we have the theorem.
\end{proof}
We get the following obvious conclusion relating orthogonal matrices of size $2$ having determinant $-1$ and a subset of $\Sym_2^0$.
\begin{corollary}\label{equality of some orthogonal matrices and some symmetric matrices of trace zero}
Any $2\times 2$ orthogonal matrix $A$ of determinant $-1$ is a symmetric matrix of trace zero.  Moreover, given any  $2\times 2$ symmetric matrix $A$ of trace zero, there exists a $2\times 2$ orthogonal matrix $B$ of determinant $-1$ and a scalar $\lambda$ such that $A=\lambda B$.  Furthermore, the set of all trace zero symmetric matrices of size $2$ having eigenvalues $1$ and $-1$ is same as the set of all size $2$ orthogonal matrices having determinant $-1$.
\end{corollary} 
\begin{proof}
Follows directly from Proposition \ref{structure of orthogonal matrices of size 2} and Theorem \ref{structure of Sym^0_2}.
\end{proof} 
We now resume the discussion about the map $T_{\theta}^{\lambda}$.  We did some calculations and obtained the coordinates of the point of a given point with respect to the line $L_{\theta}$.  Look at the same calculations from a different point of view.  Consider the matrix $\mathcal{R}_{\theta}=\big(\begin{smallmatrix}
 \cos(2\theta) &  \sin(2\theta)\\
 \sin(2\theta) &  -\cos(2\theta)
\end{smallmatrix}\big)$.  Then 
\begin{equation}\label{reflection_in terms of matrix}
\begin{pmatrix}
 \cos(2\theta) &  \sin(2\theta)\\
 \sin(2\theta) &  -\cos(2\theta)
\end{pmatrix}\begin{pmatrix}
x_0\\y_0
\end{pmatrix}=\begin{pmatrix}
x_0\;\cos(2\theta)+y_0\;\sin(2\theta)\\
x_0\;\sin(2\theta)-y_0\;\cos(2\theta)
\end{pmatrix}.
\end{equation}
Now,
\begin{equation}\label{abscissa_in terms of matrix}
\begin{split}
x_0\;\cos(2\theta)+y_0\;\sin(2\theta)=&x_0\;\cos(2\theta)+x_0\;\tan(\alpha)\sin(2\theta)(\text{as}\; (x_0,y_0)\;\text{lies on}\; L_{\alpha})\\
&=x_0\;\sec(\alpha)(\cos(\alpha)\cos(2\theta)+\sin(\alpha)\sin(2\theta))\\
&=x_0\;\sec(\alpha)\cos(2\theta-\alpha)\\
&=x_1 \;\; (\text{by}\;\eqref{abscissa of reflection}).
\end{split}
\end{equation}
Similarly, we have
\begin{equation}\label{ordinate_in terms of matrix}
\begin{split}
x_0\;\sin(2\theta)-y_0\;\cos(2\theta)&=y_0\;\cot(\alpha)\sin(2\theta)-y_0\;\cos(2\theta)(\text{as}\; (x_0,y_0)\;\text{lies on}\; L_{\alpha})\\
&=y_o\;\cosec(\alpha)(\cos(\alpha)\sin(2\theta)-\cos(2\theta)\sin(\alpha))\\
&=y_o\;\cosec(\alpha)\sin(2\theta-\alpha)\\
&=y_1 \;\; (\text{by}\;\eqref{ordinate of reflection}).
\end{split}
\end{equation}
Therefore by \eqref{reflection_in terms of matrix}, \eqref{abscissa_in terms of matrix} and \eqref{ordinate_in terms of matrix}, we have
\begin{equation}\label{reflection_as elements of Sym_2^0}
\mathcal{R}_{\theta}\begin{pmatrix}
x_0\\y_0
\end{pmatrix}=\begin{pmatrix}
 \cos(2\theta) &  \sin(2\theta)\\
 \sin(2\theta) &  -\cos(2\theta)
\end{pmatrix}\begin{pmatrix}
x_0\\y_0
\end{pmatrix}=\begin{pmatrix}
x_1\\y_1
\end{pmatrix}.
\end{equation}
Recall the definition of the map $T_{\theta}^{\lambda}$ as defined in \eqref{geometric definition of elements of Sym_0^2}.  Now treating points of $\mathbb{R}^2$ as column vectors, from \eqref{reflection_as elements of Sym_2^0} we can conclude that 
\begin{equation*}
T_{\theta}^{1}=R_{\theta}=\mathcal{R}_{\theta}.
\end{equation*} 
Moreover,
\begin{equation}\label{finally_geometry of Sym_2^0}
T_{\theta}^{\lambda}=\lambda R_{\theta}=\lambda\mathcal{R}_{\theta}.
\end{equation} 
Thus, \eqref{finally_geometry of Sym_2^0} provides the geometry of the elements of $\Sym_0^2$.  Also from now on we simply use the notation $R$ to mean both $R_{\theta}$ and $\mathcal{R}_{\theta}$ if there is there is no confusion regarding the line of reflection $L_{\theta}$.  So, $\lambda R$ and $T_{\theta}^{\lambda}$ denote the same map and we use them interchangeably.   

We now answer the questions we posed related to the map $T_{\theta}^{\lambda}$.  Towards that, we have the following proposition. 

\begin{proposition}\label{R^n is identity conditions}
Let $n$ be any positive integer.  Then 
$(\lambda R)^n=\mathbb{I}_2$ if and only if $n$ is even and $\lambda=1$ or $-1$.
\end{proposition}
\begin{proof}
It is easy to observe that $R^2=\mathbb{I}_2$.  Therefore, if $n$ is even and $\lambda=1$ or $-1$, then $(\lambda R)^n=\lambda^n R^n=\lambda^n \mathbb{I}_2=\mathbb{I}_2$.

Conversely, let $(\lambda R)^n=\mathbb{I}_2$.  Therefore,  
\begin{equation}\label{R^n is identity_converse part}
R^n=\tfrac{1}{\lambda^n}\mathbb{I}_2.
\end{equation} 
Recall that $R=\big(\begin{smallmatrix}
 \cos(2\theta) &  \sin(2\theta)\\
 \sin(2\theta) &  -\cos(2\theta)
\end{smallmatrix}\big)$.  Now if $n$ is odd, then  \eqref{R^n is identity_converse part} implies that $\cos(2\theta)=-\cos(2\theta)$, which in turn says that $\cos(2\theta)=0$.  This contradicts \eqref{R^n is identity_converse part}.  Therefore if $(\lambda R)^n=\mathbb{I}_2$ then $n$ can't be odd.  Hence $n$ is even and moreover by \eqref{R^n is identity_converse part} $\lambda=1$ or $-1$.  
\end{proof}

Now we are in a position to find some exclusive conditions which will force the orbit $\Or_{(a,b)}(T_{\theta}^{\lambda})$ to be finite.  Precisely, we obtain the following. 

\begin{theorem}\label{orbit finiteness_NASC conditions}
Let $\lambda$ be any non-zero real number and let $(a,b)$ be any point of $\mathbb{R}^2$ other than origin.  Then $\Or_{(a,b)}(T_{\theta}^{\lambda})$ is finite if and only if $\lambda=1$ or $-1$.
\end{theorem}
\begin{proof}
We first show that to get the finiteness of the orbit $\Or_{(a,b)}(T_{\theta}^{\lambda})$ it is sufficient to have $\lambda=1$ or $-1$.  For that, we proceed contra positively.
\begin{equation*}
\begin{split}
\Or_{(a,b)}(T_{\theta}^{\lambda})\; \text{is not finite}&\Rightarrow (\lambda R)^n((a,b))\neq (\lambda R)^m((a,b))\\&\;\text{for all non-negative}\; m \;\&\; n \\
&\Rightarrow (\lambda R)^{n-m}((a,b))\neq (a,b) \;(\text{assuming}\;  n>m\;\text{w.l.o.g})\\
&\Rightarrow (\lambda R)^p\neq \mathbb{I}_2 \;\text{for all even}\; p\\
&\Rightarrow \lambda\neq 1,-1\;(\text{by Proposition}\;\ref{R^n is identity conditions}).
\end{split}
\end{equation*}
We now prove the converse part.  Let $\Or_{(a,b)}(T_{\theta}^{\lambda})$ is finite.  Then
\begin{equation}\label{orbit finiteness_converse part_when point is on line of reflection}
\begin{split}
\Or_{(a,b)}(T_{\theta}^{\lambda})\; \text{is finite}&\Rightarrow (\lambda R)^n((a,b))=(\lambda R)^m((a,b))\\&\;\text{for some non-negative}\; m \;\&\; n \\
&\Rightarrow (\lambda R)^{n-m}((a,b))=(a,b) \;(\text{assuming}\;  n>m\;\text{w.l.o.g})\\
&\Rightarrow (\lambda R)^p((a,b))=(a,b),  \;\text{where}\; n-m=p\;(\text{say}).
\end{split}
\end{equation}
Now if $p$ is even, then by \eqref{orbit finiteness_converse part_when point is on line of reflection} we further have the following implications.
\begin{equation*}\label{orbit finiteness_even case}
\begin{split}
\lambda^p R((a,b))=(a,b)&\Rightarrow \lambda^p \mathbb{I}_2((a,b))=(a,b)\\&\Rightarrow \lambda^p=1\Rightarrow \lambda=1,-1.
\end{split}
\end{equation*}
For $p$ odd, we further consider two mutually exclusive and exhaustive cases.

\textit{First Case - }$(a,b)\in L_{\theta}$ : \\
By \eqref{orbit finiteness_converse part_when point is on line of reflection} we have the following implications.
\begin{equation*}
\begin{split}
\lambda^p R((a,b))=(a,b)&\Rightarrow \lambda^p (a,b)=(a,b)\\& \Rightarrow \lambda^p=1\Rightarrow \lambda=1.
\end{split}
\end{equation*}
\textit{Second Case - }$(a,b)\notin L_{\theta}$ : \\
Let $(a,b)\in L_{\alpha}$ for some real number $\alpha$. As $p$ is odd, by \eqref{orbit finiteness_converse part_when point is on line of reflection}, we have that $(\lambda R)^n((a,b))=(\lambda R)^m((a,b))$, where one of $n$ and $m$ is odd and other is even.  Without loss of generality, take $n$ to be odd and $m$ to be even.  Then we have $\lambda^n R((a,b))=\lambda^m (a,b)$.  Therefore, $R((a,b))$ and $(a,b)$ are collinear, which is a contradiction to the assumption that $(a,b)\notin L_{\theta}$ by Remark \ref{points on line of reflection are invariant under reflection}.  So, this case is not at all feasible.
\end{proof}

\begin{remark}\label{remark on finite orbits}
\begin{enumerate}
\item Theorem \ref{orbit finiteness_NASC conditions} talks about a necessary and sufficient condition for finiteness of the orbit $\Or_{(a,b)}(T_{\theta}^{\lambda})$ for non-zero values of $\lambda$ and for points $(a,b)\neq (0,0)$.  Also, $\Or_{(a,b)}(T_{\theta}^{\lambda})$ is finite if either $(a,b)=(0,0)$ or $\lambda=0$.
\item We now check when $\Or_{(a,b)}(T_{\theta}^{\lambda})$ is either singleton or consists of two elements.  
\begin{enumerate}
\item If $(a,b)=(0,0)$, then $\Or_{(0,0)}(T_{\theta}^{\lambda})=\{(0,0)\}$ for all $\theta$ and $\lambda$.
\item If $(a,b)\neq(0,0)$, then we have the following: 
\begin{enumerate}
\item If $(a,b)\in L_{\theta}$, then $\Or_{(a,b)}(T_{\theta}^{1})=\{(a,b)\}$.
\item If $(a,b)\notin L_{\theta}$, then $\Or_{(a,b)}(T_{\theta}^{1})=\{(a,b),T_{\theta}^{1}(a,b)\}$ and \;\;\;\;  $\Or_{(a,b)}(T_{\theta}^{0})$ $=\{(a,b),(0,0)\}$.
\item $\Or_{(a,b)}(T_{\theta}^{-1})=\{(a,b),T_{\theta}^{-1}(a,b)\}$.
\end{enumerate}
So, we discussed a few cases where the orbit $\Or_{(a,b)}(T_{\theta}^{\lambda})$ is either singleton or consists of two elements only.  It can be easily checked that these are all possible such orbits. 
\end{enumerate}
\end{enumerate}
\end{remark}

We now look at the sequence $\{(T_{\theta}^{\lambda})^n((a,b))\}_{n\geq 0}$.  We want to find some suitable $(a,b)\in \mathbb{R}^2$, $\theta$ and $\lambda$ such that the sequence $\{(T_{\theta}^{\lambda})^n((a,b))\}_{n\geq 0}$ is eventually constant.  At this point, we note that whenever the sequence $\{(T_{\theta}^{\lambda})^n((a,b))\}_{n\geq 0}$ is eventually constant, the orbit $\Or_{(a,b)}(T_{\theta}^{\lambda})$ must be finite.  Therefore, $\lambda$ can have only three values, namely $1$, $-1$ and $0$ unless we take $(a,b)=(0,0)$ (cf. Theorem \ref{orbit finiteness_NASC conditions} and Remark \ref{remark on finite orbits}).  If we take $(a,b)=(0,0)$, then for any $\lambda$ and $\theta$, $\{(T_{\theta}^{\lambda})^n((0,0))\}_{n\geq 0}$ is a constant sequence having all the terms equal to $(0,0)$.  So look at points other than origin.  Let $(a,b)(\neq (0,0))$ lies in the line $L_{\alpha}$ for some $\alpha\in \mathbb{R}$.  Then clearly the sequence $\{(T_{\theta}^{0})^n((a,b))\}_{n\geq 0}$ is eventually constant.  For $\lambda=1$, $\{(T_{\theta}^{1})^n((a,b))\}_{n\geq 0}$ is eventually constant only when $\theta=\alpha$.  Finally $\lambda=-1$ provides us eventually constant sequences only in the most simplest case, that is, when $(a,b)=(0,0)$.  As in the discrete topology on $\mathbb{R}^2$, only convergent sequences are the eventually constant ones, therefore the sequence $\{(T_{\theta}^{\lambda})^n((a,b))\}_{n\geq 0}$ is convergent in discrete topology in the following cases.  The point of convergence is also given accordingly. 
\begin{enumerate}
\item When $(a,b)=(0,0)$, the sequence $\{(T_{\theta}^{\lambda})^n((0,0))\}_{n\geq 0}$ converges to $(0,0)$ in discrete topology, for any $\lambda$ and $\theta$.
\item When $\lambda=0$, the sequence $\{(T_{\theta}^{0})^n((a,b))\}_{n\geq 0}$ converges to $(0,0)$ in discrete topology, for any $(a,b)$ and $\theta$.
\item When $\lambda=1$, the sequence $\{(T_{\theta}^{1})^n((a,b))\}_{n\geq 0}$ converges to $(a,b)$ in discrete topology for any $\theta$ if $(a,b)\in L_{\theta}$.
\end{enumerate}

As discrete topology is finer than usual topology on $\mathbb{R}^2$, the mentioned sequences are convergent in usual topology as well.  We now find out whether there are some more sequences $\{(T_{\theta}^{\lambda})^n((a,b))\}_{n\geq 0}$ that are convergent in usual topology.  As the metric induced by the usual topology on $\mathbb{R}^2$ is a complete metric, to find out the convergent sequences it is enough to find the Cauchy sequences over there.

We observe that for any $(x_0,y_0)\in \mathbb{R}^2$, if for some $n$, $(T_{\theta}^{\lambda})^n((x_0,y_0))\in L_{\alpha}$, then $(T_{\theta}^{\lambda})^{n+1}((x_0,y_0))\in L_{2\theta-\alpha}$.  As the lines $L_{\alpha}$ and $L_{2\theta-\alpha}$ intersect only at origin, the distance between any two consecutive terms of the sequence $\{(T_{\theta}^{\lambda})^n((x_0,y_0))\}_{n\geq 0}$ can be made arbitrarily small, that is, the sequence $\{(T_{\theta}^{\lambda})^n((x_0,y_0))\}_{n\geq 0}$ can be made to be a Cauchy sequence, only when the terms of the sequence approach origin.  So we conclude that the only possible point of convergence of the sequence $\{(T_{\theta}^{\lambda})^n((x_0,y_0))\}_{n\geq 0}$, with infinitely many distinct terms, is $(0,0)$.  We claim that this only possibility can be attained by the sequence $\{(T_{\theta}^{\lambda})^n((x_0,y_0))\}_{n\geq 0}$ for all non-zero $\lambda$ with $|\lambda|< 1$.  

Recall that $(\lambda R)^n={\lambda}^n\mathbb{I}_2$ when $n$ is even and $(\lambda R)^n={\lambda}^n R$ when $n$ is odd.  Using this we calculate some upper bounds of the distances  between terms of the sequence $\{(T_{\theta}^{\lambda})^n((x_0,y_0))\}_{n\geq 0}$.  Here by distance between any two points $x$ and $y$ of $\mathbb{R}^2$, we mean the Euclidean distance between them and denote that by $d(x,y)$. 

\textit{Case 1} : When both $n$ and $m$ are even
\begin{equation}\label{terms can be made closer_both even case}
\begin{split}
d((T_{\theta}^{\lambda})^n((x_0,y_0)),&(T_{\theta}^{\lambda})^m((x_0,y_0)))=d((\lambda R)^n,(\lambda R)^m)\\&=d(({\lambda}^n x_0,{\lambda}^n y_0),({\lambda}^m x_0,{\lambda}^m y_0))\\&\leq d(({\lambda}^n x_0,{\lambda}^n y_0),(0,0))+d((0,0),({\lambda}^m x_0,{\lambda}^m y_0))\\&=(|\lambda|^n+|\lambda|^m)\sqrt{x_0^2+y_0^2}.
\end{split}
\end{equation}

\textit{Case 2} : When both $n$ and $m$ are odd
\begin{equation}\label{terms can be made closer_both odd case}
\begin{split}
d((T_{\theta}&^{\lambda})^n((x_0,y_0)),(T_{\theta}^{\lambda})^m((x_0,y_0)))=d((\lambda R)^n,(\lambda R)^m)\\&=d(({\lambda}^n (x_0\cos(2\theta)+y_0\sin(2\theta)),{\lambda}^n (x_0\sin(2\theta)-y_0\cos(2\theta))),\\&({\lambda}^m (x_0\cos(2\theta)+y_0\sin(2\theta)),{\lambda}^m (x_0\sin(2\theta)-y_0\cos(2\theta))))\\&\leq  d(({\lambda}^n (x_0\cos(2\theta)+y_0\sin(2\theta)),{\lambda}^n (x_0\sin(2\theta)-y_0\cos(2\theta))),(0,0))\\&+d((0,0),({\lambda}^m (x_0\cos(2\theta)+y_0\sin(2\theta)),{\lambda}^m (x_0\sin(2\theta)-y_0\cos(2\theta))))\\&=(|\lambda|^n+|\lambda|^m)\sqrt{x_0^2+y_0^2}.
\end{split}
\end{equation}

\textit{Case 3} : When $n$ is even and $m$ is odd 
\begin{equation}\label{terms can be made closer_even and odd case}
\begin{split}
&d((T_{\theta}^{\lambda})^n((x_0,y_0)),(T_{\theta}^{\lambda})^m((x_0,y_0)))=d((\lambda R)^n,(\lambda R)^m)\\&=d(({\lambda}^n x_0,{\lambda}^n y_0),({\lambda}^m (x_0\cos(2\theta)+y_0\sin(2\theta)),{\lambda}^m (x_0\sin(2\theta)-y_0\cos(2\theta))))\\&\leq d(({\lambda}^n x_0,{\lambda}^n y_0),(0,0))\\&+d((0,0),({\lambda}^m (x_0\cos(2\theta)+y_0\sin(2\theta)),{\lambda}^m (x_0\sin(2\theta)-y_0\cos(2\theta))))\\&=(|\lambda|^n+|\lambda|^m)\sqrt{x_0^2+y_0^2}.
\end{split}
\end{equation}
So given any $\epsilon >0$, by \eqref{terms can be made closer_both even case}, \eqref{terms can be made closer_both odd case} and \eqref{terms can be made closer_even and odd case}, we can choose $n_0\in \mathbb{N}$ such that for all $m,n\geq n_0$, $d((T_{\theta}^{\lambda})^n((x_0,y_0)),(T_{\theta}^{\lambda})^m((x_0,y_0)))\leq (|\lambda|^n+|\lambda|^m)\sqrt{x_0^2+y_0^2}< \epsilon$ whenever $|\lambda|<1$.  Therefore, the sequence $\{(T_{\theta}^{\lambda})^n((x_0,y_0))\}_{n\geq 0}$ is a Cauchy sequence for all $\lambda$ with $|\lambda|<1$ and hence convergent to $(0,0)$ in usual topology as we justified earlier that $(0,0)$ is the only possible point of convergence. 
\begin{figure}
\centering
\includegraphics[scale=.45]{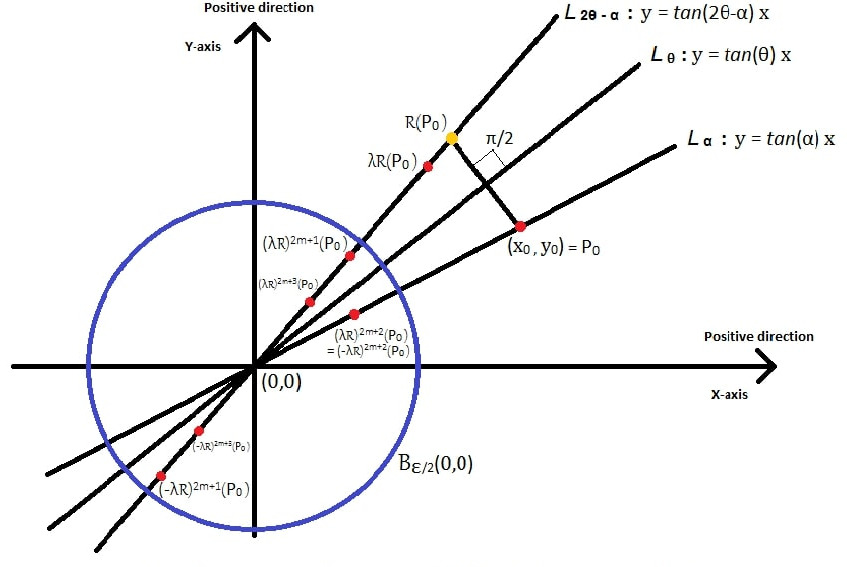}
\caption{Convergence of the sequence $\{(\lambda R)^n(P_0)\}_{n\geq 0}$ for $0<\mid \lambda \mid <1$} \label{sample-figure_2}
\end{figure} 

Also, it can be directly checked that the sequence $\{(T_{\theta}^{\lambda})^n((x_0,y_0))\}_{n\geq 0}$ converges to $(0,0)$ for all $\lambda$ with $|\lambda|<1$.

\textit{Case 1} : When $n$ is even
\begin{equation}\label{distance from origin can be made enough small_even case}
\begin{split}
d((T_{\theta}^{\lambda})^n((x_0,y_0)),(0,0))&=d((\lambda R)^n,(0,0))\\&=d(({\lambda}^n x_0,{\lambda}^n y_0),(0,0))\\&=|\lambda|^n\sqrt{x_0^2+y_0^2}.
\end{split}
\end{equation}

\textit{Case 2} : When $n$ is odd
\begin{equation}\label{distance from origin can be made enough small_odd case}
\begin{split}
d((T_{\theta}^{\lambda})^n&((x_0,y_0)),(0,0))=d((\lambda R)^n,(0,0))\\&=d(({\lambda}^n (x_0\cos(2\theta)+y_0\sin(2\theta)),{\lambda}^n (x_0\sin(2\theta)-y_0\cos(2\theta))),(0,0))\\&=|\lambda|^n\sqrt{x_0^2+y_0^2}.
\end{split}
\end{equation}

So given any $\epsilon >0$, by \eqref{distance from origin can be made enough small_even case} and \eqref{distance from origin can be made enough small_odd case}, we can choose $n_0\in \mathbb{N}$ such that for all $n\geq n_0$, $d((T_{\theta}^{\lambda})^n((x_0,y_0)),(0,0))=|\lambda|^n\sqrt{x_0^2+y_0^2}<\tfrac{\epsilon}{2}$ whenever $|\lambda|<1$.  Therefore, for all $\lambda$ with $|\lambda|<1$, the terms of the sequence $\{(T_{\theta}^{\lambda})^n((x_0,y_0))\}_{n\geq 0}$ eventually lie in the open ball $B_{\tfrac{\epsilon}{2}}(0,0)$ of radius $\tfrac{\epsilon}{2}$ and having centre at $(0,0)$ and hence convergent to $(0,0)$ in usual topology.

We want to summarize what we have discussed so far regarding the convergence of the sequence $\{(T_{\theta}^{\lambda})^n((x_0,y_0))\}_{n\geq 0}$ from a different perspective.  For that we lend some terminologies from dynamical systems and define those solely in our context.  
\begin{definition}
Let $X$ be a finite dimensional vector space over $\mathbb{R}$ and $d$ be a metric on $X$.  Also let $L : X\rightarrow X$ be a linear map.  Then any two points $x_1$, $x_2$ of $X$ are said to be forward asymptomatic with respect to the map $L$ if $d(L^n(x_1),L^n(x_2))\longrightarrow 0$ as $n\longrightarrow \infty$.  For any $x\in X$, the stable set of $x$ with respect to $L$, denoted by $\Stab_L(x)$, is the set of all points forward asymptomatic to $x$ with respect to $L$.
\end{definition}
We now have the following theorem regarding how the stable set of any point of $\mathbb{R}^2$ with respect to any arbitrary element $\lambda R$ of $\Sym_2^0$ changes all of a sudden as $\lambda$ varies from $|\lambda|<1$ to $|\lambda|\geq 1$.
\begin{theorem}\label{stable set}
Consider the metric space $(\mathbb{R}^2,d)$ where $d$ denotes the usual metric.  Then the following hold for any point $(a,b)\in \mathbb{R}^2$:
\begin{enumerate}
\item For any $\lambda$ with $|\lambda|<1$, 
\begin{equation*}
\Stab_{T_{\theta}^{\lambda}}((a,b))=\mathbb{R}^2.
\end{equation*}
\item For any $\lambda$ with $|\lambda|\geq 1$, 
\begin{equation*}
\Stab_{T_{\theta}^{\lambda}}((a,b))=\{(a,b)\}.
\end{equation*}
\end{enumerate}
\end{theorem}
\begin{proof}
For any $(c,d)\in \mathbb{R}^2$ and for any odd $n$,
\begin{equation}\label{stable set_odd case}
\begin{split}
&d((T_{\theta}^{\lambda})^n((a,b)),(T_{\theta}^{\lambda})^n((c,d)))\\&=d(({\lambda}^n (a\;\cos(2\theta)+b\;\sin(2\theta)),{\lambda}^n (a\; \sin(2\theta)-b\;\cos(2\theta))),\\&({\lambda}^n (c\;\cos(2\theta)+d\;\sin(2\theta)),{\lambda}^n (c\;\sin(2\theta)-d\;\cos(2\theta))))\\
&=|\lambda|^n\sqrt{((a-c)\cos(2\theta)+(b-d)\sin(2\theta))^2+((a-c)\sin(2\theta)-(b-d)\cos(2\theta))^2}\\&=|\lambda|^n \sqrt{(a-c)^2+(b-d)^2}\\&=|\lambda|^n d((a,b),(c,d)).
\end{split}
\end{equation}
Also for any $(c,d)\in \mathbb{R}^2$ and even $n$,
\begin{equation}\label{stable set_even case}
\begin{split}
d(&(T_{\theta}^{\lambda})^n((a,b)),(T_{\theta}^{\lambda})^n((c,d)))=d((\lambda^n a,\lambda^n b),(\lambda^n c,\lambda^n d))\\&=|\lambda|^n \sqrt{(a-c)^2+(b-d)^2}=|\lambda|^n d((a,b),(c,d)).
\end{split}
\end{equation}
Therefore we have the following as $\lambda$ varies.
\begin{enumerate}
\item For all $(c,d)\in \mathbb{R}^2$, $d((T_{\theta}^{\lambda})^n((a,b)),(T_{\theta}^{\lambda})^n((c,d)))\longrightarrow  0$ as $n\longrightarrow \infty$ by \eqref{stable set_odd case} and \eqref{stable set_even case}, whenever $|\lambda|<1$.  So any point of $\mathbb{R}^2$ is forward asymptomatic to $(a,b)$ and hence $\Stab_{T_{\theta}^{\lambda}}((a,b))=\mathbb{R}^2$.
\item By \eqref{stable set_odd case} and \eqref{stable set_even case}, for $|\lambda| \geq 1$, $d((T_{\theta}^{\lambda})^n((a,b)),(T_{\theta}^{\lambda})^n((c,d)))\longrightarrow 0$ as $n\longrightarrow \infty$ if and only if $d((a,b),(c,d))=0$ if and only if $(a,b)=(c,d)$.  So the only point forward asymptomatic to $(a,b)$ is $(a,b)$ itself, that is to say $\Stab_{T_{\theta}^{\lambda}}((a,b))=\{(a,b)\}$.
\end{enumerate}
\end{proof}
\begin{remark}
First part of Theorem \ref{stable set} can be proved alternatively as follows. We first claim that $\Stab_{T_{\theta}^{\lambda}}((0,0))=\mathbb{R}^2$ for any $\lambda$ with $|\lambda|<1$.  When $\lambda=0$, it is easy to see that any point of $\mathbb{R}^2$ is forward asymptomatic to $(0,0)$ with respect to $T_{\theta}^{\lambda}$.  Also, the same is true for any $\lambda$ with $0<|\lambda|<1$ by \eqref{distance from origin can be made enough small_even case} and \eqref{distance from origin can be made enough small_odd case}.  Hence we have the claim.  Now for any point $(a,b)\neq (0,0)$,
\begin{equation*}
\begin{split}
d(&(T_{\theta}^{\lambda})^n((a,b)),(T_{\theta}^{\lambda})^n((c,d)))\\
&\leq d((T_{\theta}^{\lambda})^n((a,b)),(T_{\theta}^{\lambda})^n((0,0)))+ d((T_{\theta}^{\lambda})^n((0,0)),(T_{\theta}^{\lambda})^n((c,d)))\\&=d((T_{\theta}^{\lambda})^n((a,b)),(0,0))+ d((0,0),(T_{\theta}^{\lambda})^n((c,d)))
\end{split}
\end{equation*}
Therefore for all $(c,d)\in \mathbb{R}^2$, $d((T_{\theta}^{\lambda})^n((a,b)),(T_{\theta}^{\lambda})^n((c,d)))\longrightarrow  0$ as $n\longrightarrow \infty$ by \eqref{distance from origin can be made enough small_even case} and \eqref{distance from origin can be made enough small_odd case}.  So any point of $\mathbb{R}^2$ is forward asymptomatic to $(a,b)$ and hence $\Stab_{T_{\theta}^{\lambda}}((a,b))=\mathbb{R}^2$.
\end{remark}
\section{Applications of the structure of $\S\lowercase{ym}_2^0$}\label{section 3}
In this section we talk about a couple of applications of the obtained structure of $\Sym_2^0$.  In the first subsection we obtain the structure of the set $\PSym_2^0$.  In the next subsection we reinterpret some trigonometric formulas and show how those can be used to solve purely geometric questions. 
\subsection{The structure of $\PSym_2^0$}\label{subsection 3.1} 
Recall that $\PSym^0_n$ is the set of all $n\times n$ symmetric matrices, trace of whose product with any element of $\Sym^0_n$ is zero.  In this subsection we obtain the structure of $\PSym^0_2$ in the form of following theorem.
\begin{theorem}\label{main theorem_special case}
The set $\PSym^0_2$ is the set of all scalar matrices of size $2$.
\end{theorem}
\begin{proof}
Let $A=\big(\begin{smallmatrix}a & b\\b & c \end{smallmatrix}\big)\in \PSym^0_2$ be arbitrarily chosen.  Then $\Tr(B A)=0$ for all $B\in \Sym^0_2$.  Therefore by Theorem \ref{structure of Sym^0_2}, we have:
\begin{equation*}
\Tr(\big(\begin{smallmatrix}\gamma\;\cos (\theta) & \gamma\;\sin (\theta)\\
\gamma\;\sin (\theta) & -\gamma\;\cos (\theta) \end{smallmatrix}\big) \big(\begin{smallmatrix}a & b\\b & c \end{smallmatrix}\big))=0,
\end{equation*}
for all $\gamma\in \mathbb{R}$ and $\theta\in [0,2\pi)$.  That is, for all $\gamma\in \mathbb{R}$ and $\theta\in [0,2\pi)$, 
\begin{equation}\label{for plugging values to prove the main theorem for n=2}
\gamma(a\;\cos(\theta)+2b\;\sin(\theta)-c\;\cos(\theta))=0.
\end{equation}
Plugging $\gamma=1$ and $\theta=0$ in \eqref{for plugging values to prove the main theorem for n=2}, we get $a=c$.  Similarly, plugging $\gamma=1$ and $\theta=\tfrac{\pi}{2}$ in \eqref{for plugging values to prove the main theorem for n=2}, we get $b=0$.  Hence, the assertion follows.
\end{proof} 
In fact, we have the following more general result.  This is motivated by \cite[Problem 9, Exercises VI, \S 2, p.190]{La}. 
\begin{theorem}\label{main theorem_with more prerequisite}
For any positive integer $n\geq 2$, $\PSym^0_n$ is the set of all scalar matrices of size $n$.
\end{theorem}
\begin{proof}
Consider the Frobenius inner product $\langle , \rangle$ on $\Sym_n$ given by $\langle A, B \rangle = \Tr(AB)$, for any $A, B \in \Sym_n$.  Then $\PSym^0_n$ is nothing but the orthogonal complement $({\Sym_n^0})^{\perp}$ of $\Sym_n^0$ with respect to the Frobenius inner product.  Now as $\Sym_n^0$ is a subspace of $\Sym_n$ of co-dimension $1$ and as $\mathbb{I}_n\in \PSym^0_n$, we have the theorem.   
\end{proof}
So, the proof of Theorem \ref{main theorem_special case} incorporates the obtained structure of $\Sym_2^0$ and provides an alternative way of proving Theorem \ref{main theorem_with more prerequisite} for $n=2$. 
\subsection{Interpreting some trigonometric formulas}\label{subsection 3.2}
In this subsection, we prove an interesting geometric property of $\mathbb{R}^2$ using the geometric interpretation of the elements of $\Sym_2^0$ (cf. \eqref{finally_geometry of Sym_2^0}).  Precisely, we prove the following : 
\begin{theorem}\label{reinterpreting some trigonometric formulas}
Given any real number $\beta\in [0,\pi)$, let $L_{\beta}$ denotes the line passing through origin and making an angle $\beta$ (in the anticlockwise direction) with positive direction of $x$ axis.  Then, given any point $(x_0,y_0)$ of $\mathbb{R}^2$, rotating the point of reflection of $(x_0,y_0)$ clockwise (respectively anticlockwise) with respect to the line $L_{\theta/2}$ by an angle $\alpha$ is same as reflecting $(x_0,y_0)$ with respect to the line $L_{(\theta-\alpha)/2}$ (respectively $L_{(\theta+\alpha)/2}$).  
\end{theorem}
\begin{proof}
We interpret a point $(x_0,y_0)$ of $\mathbb{R}^2$ as the $2\times 1$ column vector $\big(\begin{smallmatrix}x_0 \\y_0 \end{smallmatrix}\big)$.  From \eqref{geometric definition of elements of Sym_0^2}, \eqref{finally_geometry of Sym_2^0} and Theorem \ref{structure of Sym^0_2}, we have that an element $\big(\begin{smallmatrix} \cos(\theta) & \sin(\theta)\\
\sin(\theta) & -\cos(\theta) \end{smallmatrix}\big)$ of $\Sym^0_2$, denoted by $R_{\theta /2}$, reflects a point $(x_0,y_0)$ with respect to the line $L_{\theta /2}$.  An orthogonal matrix $\big(\begin{smallmatrix} \cos(\alpha) & \sin(\alpha)\\
-\sin(\alpha) & \cos(\alpha) \end{smallmatrix}\big)$ of determinant $1$, denoted by $O^1_{\alpha}$, represents clockwise rotation of the plane $\mathbb{R}^2$.  Similarly, the matrix $O^1_{-\alpha}$ represents anticlockwise rotation of the plane $\mathbb{R}^2$  (cf. \cite[p.348]{He}).  Hence, to prove the theorem, we need to show that $O^1_{\alpha}\cdot R_{\theta}\big(\begin{smallmatrix}x_0 \\y_0 \end{smallmatrix}\big)=R_{\theta-\alpha}\big(\begin{smallmatrix}x_0 \\y_0 \end{smallmatrix}\big)$ for clockwise case and $O^1_{-\alpha}\cdot R_{\theta}\big(\begin{smallmatrix}x_0 \\y_0 \end{smallmatrix}\big)=R_{\theta+\alpha}\big(\begin{smallmatrix}x_0 \\y_0 \end{smallmatrix}\big)$ for anticlockwise case.  Now,
\begin{equation}\label{geometry of trigonometric equations_1}
\begin{split}
O^1_{\alpha}\cdot R_{\theta}=&
\begin{pmatrix}
\cos(\alpha) & \sin(\alpha)\\
-\sin(\alpha) & \cos(\alpha)
\end{pmatrix}\begin{pmatrix}
\cos(\theta) & \sin(\theta)\\
\sin(\theta) & -\cos(\theta)
\end{pmatrix}\\&=\begin{pmatrix}
\cos(\alpha)\cos(\theta)+\sin(\alpha)\sin(\theta) & \cos(\alpha)\sin(\theta)-\sin(\alpha)\cos(\theta)\\
\cos(\alpha)\sin(\theta)-\sin(\alpha)\cos(\theta) & -\cos(\alpha)\cos(\theta)-\sin(\alpha)\sin(\theta)
\end{pmatrix}\\&=\begin{pmatrix}
\cos(\theta-\alpha) & \sin(\theta-\alpha)\\
\sin(\theta-\alpha) & -\cos(\theta-\alpha)
\end{pmatrix}=R_{\theta-\alpha}.
\end{split}
\end{equation}
Similarly we have :
\begin{equation}\label{geometry of trigonometric equations_2}
O^1_{-\alpha}\cdot R_{\theta}=R_{\theta+\alpha}.
\end{equation}
The theorem now follows from \eqref{geometry of trigonometric equations_1} and \eqref{geometry of trigonometric equations_2} for clockwise and anticlockwise scenario respectively. 
\end{proof}

\begin{remark}\label{trigonometry to geometry}
Theorem \ref{reinterpreting some trigonometric formulas} talks about a geometric property of the plane $\mathbb{R}^2$ in both clockwise and anticlockwise context.  Let's denote that property by $P\circlearrowright$ for clockwise case and by $P\circlearrowleft$ for anticlockwise case.  Now consider the following standard trigonometric equalities :
\begin{equation}\label{trigeqn1}
\begin{split}
&\cos(\theta+\alpha)=\cos(\alpha)\cos(\theta)-\sin(\alpha)\sin(\theta),\\& \sin(\theta+\alpha)= \cos(\alpha)\sin(\theta)+\sin(\alpha)\cos(\theta).
\end{split} 
\end{equation}
\begin{equation}\label{trigeqn2}
\begin{split}
&\cos(\theta-\alpha)=\cos(\alpha)\cos(\theta)+\sin(\alpha)\sin(\theta),\\
& \sin(\theta-\alpha)= \cos(\alpha)\sin(\theta)-\sin(\alpha)\cos(\theta).
\end{split}
\end{equation}
So, Theorem \ref{reinterpreting some trigonometric formulas} says that $P\circlearrowright$ can be thought of as a geometric interpretation of the equalities in \eqref{trigeqn2} (cf. \eqref{geometry of trigonometric equations_1}).  Similarly, deciphering \eqref{trigeqn1} geometrically, we obtain $P\circlearrowleft$ (cf. \eqref{geometry of trigonometric equations_2}). 
\end{remark}

\section{Conclusion}\label{section 4}
It is worth mentioning that the structure of some subsets of $\Sym^0_n$ can be obtained for $n>2$ as well by adapting the method used in Theorem \ref{structure of Sym^0_2}, provided some conditions being suitably put on the set of eigenvalues (other than they add upto zero) of all its elements.
\section{Acknowledgements}
The author wishes to thank Indian Institute of Science Education and Research Tirupati (Award No. - IISER-T/Offer/PDRF/A.M./M/01/2021) for financial support.  The author is grateful to the referee for several valuable comments.   
 

\begin{thebibliography}{99} 
\bibitem[1]{Al-Mh} Albert, A. A. and Muckenhoupt, B., On matrices of trace zeros, {\it Michigan Math. J.}, {\bf 4} (1957), 1--3.

\bibitem[2]{Fl} Fiedler, M., Eigenvalues of nonnegative symmetric matrices, {\it Linear Algebra Appl.}, {\bf 9} (1974), 119--142.

\bibitem[3]{He} Herstein, I. N., {\it Topics in algebra}, 2nd ed. Lexington: Xerox College Publishing, 1975.

\bibitem[4]{La} Lang, S., {\it Linear algebra}, Inc.:Reading: Addison-Wesley Publishing Co., 1966.

\bibitem[5]{Lw-Lo} Loewy, R. and London, D., A note on an inverse problem for nonnegative matrices, {\it Linear and Multilinear Algebra}, {\bf 6} (1978), 83--90.

\bibitem[6]{Lw-Mc} Loewy, R., and McDonald, J. J., The symmetric nonnegative inverse eigenvalue problem for $5\times 5$ matrices, {\it Linear Algebra Appl.} {\bf 393} (2004), 275--298.

\bibitem[7]{Mc-Ne} McDonald, J. J. and Neumann, M., The Soules approach to the inverse eigenvalues problem for nonnegative symmetric matrices of order $n\leq 5$, {\it Contemp. Math.}, {\bf 259} (2000), 387--407.

\bibitem[8]{Sh} Shoda, K., Einige S\"atze \"uber Matrizen, {\it Jpn. J. Math.}, {\bf 13}(3) (1937), 361--365.

\bibitem[9]{Sp} Spector, O., A characterization of trace zero symmetric nonnegative $5\times 5$ matrices, {\it Linear Algebra Appl.}, {\bf 434}(4) (2011), 1000--1017.
\end{thebibliography}
\end{document}